\documentclass[11pt]{amsart}
\usepackage{cases}

\theoremstyle{plain}
\newtheorem{theorem}                {Theorem}      [section]
\newtheorem*{theorem*}                {Theorem}
\newtheorem{proposition}  [theorem]  {Proposition}
\newtheorem{corollary}    [theorem]  {Corollary}
\newtheorem{lemma}        [theorem]  {Lemma}

\theoremstyle{definition}

\newtheorem{remark}       [theorem]  {Remark}

\DeclareMathOperator{\trace}{trace} 

\DeclareMathOperator{\Div}{div}

\DeclareMathOperator{\cst}{constant}

\DeclareMathOperator{\grad}{grad}

\DeclareMathOperator{\sol}{Sol_3}
\DeclareMathOperator{\nil}{Nil_3}
\newcommand{\vol}{v_{\Sigma}}
\numberwithin{equation}{section}

\begin{document}

\title[Simons type formulas]{Simons type formulas for surfaces in $\sol$ and applications}

\author{Dorel~Fetcu}

\thanks{}

\address{Department of Mathematics and Informatics\\
Gh. Asachi Technical University of Iasi\\
Bd. Carol I, 11 \\
700506 Iasi, Romania} \email{dorel.fetcu@etti.tuiasi.ro}



\begin{abstract} We compute the Laplacian of the squared norm of the second fundamental form of a surface in $\sol$ and then use this Simons type formula to obtain some gap results for compact constant mean curvature surfaces of this space. 
\end{abstract}

\maketitle

\section{Introduction}

Initiated by H. Rosenberg in \cite{R}, the study of minimal and constant mean curvature (CMC) surfaces in product spaces $\mathbb{M}^2\times\mathbb{R}$, with $\mathbb{M}^2$ a real space form, is, for two decades now, one of the most interesting and fast growing topic in the theory of submanifolds. Moreover, studies on the geometry of these surfaces were developed in the more general case when the ambient space is a homogeneous 3-manifold, that is, beside real space forms, one of the product spaces $\mathbb{S}^2\times\mathbb{R}$, $\mathbb{H}^2\times\mathbb{R}$, then $\widetilde{PSL}_2(\mathbb{R})$, the Heisenberg group $\nil$, and $\sol$.

Two of the most powerful tools used in these studies are Abresch-Rosenberg differentials and Simons type formulas, often used together, as we will explain in the following. 

In his seminal paper \cite{JS}, J. Simons computed the Laplacian of the squared norm of the second fundamental form of a minimal submanifold of a Riemannian manifold and then used it to characterize certain minimal submanifolds of a sphere and Euclidean space. Such equations, nowadays called Simons type formulas, were generalized for CMC hypersurfaces in space forms by K. Nomizu and B. Smyth \cite{NS} and then, by many other authors, for CMC submanifolds and submanifolds with parallel mean curvature in space forms. Almost ten years later S.~Y.~Cheng and S.~T.~Yau
~\cite{CY} proved a general formula of this type for a symmetric $(1,1)$-tensor $S$ defined on a Riemannian manifold. All these equations use the fact that the shape operator $A$ of a submanifold in a space form satisfies the classical Codazzi property $(\nabla_XA)Y=(\nabla_YA)X$ (in the Cheng-Yau formula this is one of the conditions imposed on $S$). 

However, when working in other ambient spaces, the shape operator $A$ may fail to satisfy this property and then the situation becomes more complicated, as shown, for example, in \cite{B}. In the case of most of the homogeneous $3$-manifolds this problem was solved by using another operator, obtained from the Abresch-Rosenberg differential. This differential was introduced by U.~Abresch and H.~Rosenberg in \cite{AR1,AR2} and it is the traceless part of a certain quadratic form defined on surfaces of $\mathbb{S}^2\times\mathbb{R}$, $\mathbb{H}^2\times\mathbb{R}$, $\widetilde{PSL}_2(\mathbb{R})$, and $\nil$. The differential is holomorphic if and only if the surface is CMC. By using this result, M.~Batista \cite{B} introduced an operator $S$ on a CMC surface of $\mathbb{S}^2\times\mathbb{R}$ or $\mathbb{H}^2\times\mathbb{R}$ that satisfies the classical Codazzi property, and then obtained a Simons type equation using $S$ instead of the shape operator $A$. This result was then generalized by J.~M.~Espinar and H.~A.~Trejos \cite{ET} to all other spaces where Abresch-Rosenberg differentials do exist.

Studying the geometry of surfaces in $\sol$ meets many difficulties that not appear in the case of other homogeneous $3$-manifolds (see \cite{DM}), at least when using a global approach rather than a local one. In spite of these difficulties, important results concerning, for example, the existence and uniqueness of CMC spheres \cite{DM,M,MMPR}, the totally umbilic surfaces \cite{ST}, and half-space theorems for minimal surfaces \cite{DMR}, were proved. Also, the way that a compact CMC surface with non-empty boundary inherits the symmetry of its boundary is described in \cite{L1}. Interesting results on CMC surfaces were also obtained, using a local approach and under supplementary geometric hypotheses (see, for example, \cite{L2,LM2}).

The most important problem when it comes about computing Simons type formulas for surfaces in $\sol$ is that, besides the fact that the shape operator $A$ does not have the classical Codazzi property, $\sol$ is the only homogeneous $3$-manifold where an Abresch-Rosenberg differential does not exist (see \cite{AR2}) and therefore one does not have the possibility of using an alternative operator instead of $A$. 

In this paper we develop a Simons type equation for the squared norm of the second fundamental form and a second formula also using the Laplacian of the squared inner product between the normal vector field to the surface and a special direction in $\sol$. Although these Simons type equations are quite complicated, as they involve all three vector fields of the canonical frame on $\sol$, one can derive interesting geometric information from them, e.g., Theorems \ref{thm:compact1} and \ref{thm:compact2} that describe two gap phenomena for compact CMC surfaces.

\noindent{\bf Acknowledgments.} The author is grateful to Harold Rosenberg for valuable comments that improved the paper and constant encouragement.

\noindent {\bf Conventions.} We work in the smooth category and assume surfaces to be connected and without boundary. On compact Riemannian surfaces we consider the canonical Riemannian measure. If $\Sigma^2$ is a surface of $\sol$ and $X$ is a vector field on $\sol$, then we denote by $X^{\top}$ and $X^{\perp}$ its tangent and normal parts to $\Sigma^2$, respectively.

\section{Preliminaries}

In this section we shall briefly present some known facts and results on the Lie group $\sol$ that will be used later on. Thus, $\sol$ is the $\mathbb{R}^3$ with the Riemannian metric
$$
\langle,\rangle=e^{2z}dx^2+e^{-2z}dy^2+dz^2,
$$
where $(x,y,z)$ are the canonical coordinates of $\mathbb{R}^3$.

A left-invariant orthonormal frame field $\{E_1,E_2,E_3\}$ with respect to this metric, called the canonical frame, is defined by
$$
E_1=e^{-z}\frac{\partial}{\partial x},\quad E_2=e^z\frac{\partial}{\partial y},\quad E_3=\frac{\partial}{\partial z}.
$$
The Levi-Civita connection of $\sol$ is then the following
\begin{equation}\label{nabla_sol}
\begin{array}{lll}
\bar\nabla_{E_1}E_1=-E_3,& \bar\nabla_{E_1}E_2=0,& \bar\nabla_{E_1}E_3=E_1\\ \\
\bar\nabla_{E_2}E_1=0,& \bar\nabla_{E_2}E_2=E_3,& \bar\nabla_{E_2}E_3=-E_2\\ \\
\bar\nabla_{E_3}E_1=0,& \bar\nabla_{E_3}E_2=0,& \bar\nabla_{E_3}E_3=0.
\end{array}
\end{equation}

One can see that the vertical vector field $E_3$ foliates $\sol$ by vertical geodesics. Moreover, we get that the sectional curvatures of the vertical plane fields $(E_1,E_3)$ and $(E_2,E_3)$ are equal to $-1$, while that of the horizontal plane field $(E_1,E_2)$ is equal to $1$. 

Next, from the expression of the Riemannian metric in $\sol$ it is easy to see that the leaves of the first two canonical foliations
$$
\mathcal{F}_1\equiv\{x=\cst\}\quad\textnormal{and}\quad\mathcal{F}_2\equiv\{y=\cst\},
$$
are isometric to the hyperbolic plane $\mathbb{H}^2$. These leaves are the only totally geodesic surfaces in $\sol$ (see \cite{ST}).

The leaves of the third canonical foliation $\mathcal{F}_3\equiv\{z=\cst\}$
are isometric to $\mathbb{R}^2$ with its usual flat metric and are minimal. The maximum principle for minimal surfaces then implies that there are no compact
minimal surfaces in $\sol$ (see \cite{DM,M}).

The curvature tensor $\bar R$ of $\sol$ is given by (see \cite{ST})
\begin{equation}\label{eq:barR}
\begin{array}{lcl}
\bar R(X,Y)Z&=&\langle Y,Z\rangle X-\langle X,Z\rangle Y+2\langle Z,E_3\rangle(\langle X,E_3\rangle Y-\langle Y,E_3\rangle X)\\ \\&&+2(\langle X,Z\rangle\langle Y,E_3\rangle-\langle Y,Z\rangle\langle X,E_3\rangle)E_3,
\end{array}
\end{equation}
where we use the sign convention $\bar R(X,Y)Z=\nabla_X\nabla_YZ-\nabla_Y\nabla_XZ-\nabla_{[X,Y]}Z$.

Now, let us consider a surface $\Sigma^2$ of $\sol$, with a unit vector field $\xi$ normal to $\Sigma^2$. We recall the Gauss and the Weingarten equations of the surface
$$
\bar\nabla_XY=\nabla_XY+\sigma(X,Y)\quad\textnormal{and}\quad \bar\nabla_x\xi=-AX,
$$
for all vector fields $X$ and $Y$ tangent to the surface, where $\nabla$ is the induced connection on $\Sigma^2$, $\sigma$ is its second fundamental form, and $A$ its shape operator. The mean curvature vector field of $\Sigma^2$ is given by $\vec H=f\xi$, where $f=(1/2)\trace A$ is the mean curvature function. 

If the mean curvature function $f$ is constant, the surface $\Sigma^2$ is called a \textit{constant mean curvature (CMC) surface}.

The Codazzi equation of the surface reads, for any vector fields $X$, $Y$, $Z$, tangent to $\Sigma^2$ and any normal vector field $V$,
\begin{equation}\label{eq:Codazzi_gen}
\begin{array}{cl}
\langle \bar R(X,Y)Z,V\rangle=&\langle\nabla^{\perp}_X\sigma(Y,Z),V\rangle-\langle\sigma(\nabla_XY,Z),V\rangle
-\langle\sigma(Y,\nabla_XZ),V\rangle\\ \\&-\langle\nabla^{\perp}_Y\sigma(X,Z),V\rangle+\langle\sigma(\nabla_YX,Z),V\rangle
+\langle\sigma(X,\nabla_YZ),V\rangle.
\end{array}
\end{equation}

From the Gauss equation of the surface
$$
\langle R(X,Y)Z,W\rangle=\langle\bar R(X,Y)Z,W\rangle+\langle AY,Z\rangle\langle AX,W\rangle-\langle AX,Z\rangle\langle AY,W\rangle,
$$
for all tangent vector fields $X$, $Y$, $Z$, and $W$, where $R$ is the curvature tensor of $\Sigma^2$, one gets the expression of the Gaussian curvature of the surface
\begin{equation}\label{eq:Gaussian_curvature}
K=\langle R(X_1,X_2)X_2,X_1\rangle=2\langle\xi,E_3^{\perp}\rangle^2-1+2f^2-\frac{1}{2}|A|^2,
\end{equation}
where $\{X_1,X_2\}$ is an orthonormal frame field on $\Sigma^2$ and $\xi$ is a unit normal vector field.

The literature on CMC surfaces in $\sol$, as well as in the other homogeneous $3$-manifolds, experienced a steady growth in the last fifteen years and some very interesting results were obtained in papers like, for example, \cite{DM,L1,L2,LM2,M,MMPR}. We will only mention here a beautiful result concerning the existence and uniqueness of compact CMC surfaces. B.~Daniel and P.~Mira \cite{DM} developed a new method of studying such surfaces in $\sol$ and they classified CMC spheres for values of the mean curvature greater than $1/\sqrt{3}$. Then, W.~H.~Meeks \cite{M} completed this classification (see also \cite{MMPR}). 

\begin{theorem*}\cite{DM,M}
For any $H>0$ there is a unique constant mean curvature sphere $S_{H}$ in $\sol$ with mean curvature $H$. Moreover, $S_H$ is maximally symmetric, embedded, and
has index one.
\end{theorem*}

To end this section we also recall the following important result on compact CMC surfaces in $\sol$.

\begin{theorem*}\cite{DHM,EGR}
Any compact embedded constant mean curvature surface $\Sigma^2$ in $\sol$ with mean curvature $H>0$ is, topologically, a constant mean curvature sphere $S_H$ with the same mean curvature $H$. Furthermore, after some left translation of $\Sigma^2$, the leaves $\{x = 0\}$ and $\{y = 0\}$ are planes of Alexandrov symmetry.
\end{theorem*}

\section{Simons type formulas}

In the following, we will compute the Laplacian of the squared norm of the second fundamental form of a surface in $\sol$ and also the Laplacian of the squared norm of the normal part of the vector field $E_3$. Thus we obtain two Simons type formulas, the first one of a classical form and the second, a slightly adapted one, as the result of using both Laplacians.

Let $\Sigma^2$ be a surface in $\sol$ and $\xi$ be a unit normal vector field. Let $A$ be its shape operator and $f=(1/2)\trace A$ the mean curvature function.

First, from the Codazzi equation \eqref{eq:Codazzi_gen}, one obtains
$$
\langle \bar R(X,Y)Z,\xi\rangle=\langle(\nabla_XA)Y-(\nabla_YA)X,Z\rangle,
$$
for all vector fields $X$, $Y$, and $Z$ tangent to $\Sigma^2$ and, therefore, using \eqref{eq:barR}, we obtain
\begin{equation}\label{eq:Codazzi}
(\nabla_XA)Y=(\nabla_YA)X+2\langle \xi,E^{\perp}_3\rangle(\langle Y,E^{\top}_3\rangle X-\langle X,E^{\top}_3\rangle Y).
\end{equation}

Next, we have the Weitzenb\"ock formula
\begin{equation}\label{eq:Laplacian}
\frac{1}{2}\Delta|A|^2=|\nabla A|^2+\langle\trace\nabla^2A,A\rangle,
\end{equation}
where we extended the metric $\langle,\rangle$ to the tensor space in the standard way.

We will compute the second term in the right hand side of \eqref{eq:Laplacian} by using the same method in \cite{NS}. 

Let us consider
$$
C(X,Y)=(\nabla^2 A)(X,Y)=\nabla_X(\nabla_YA)-\nabla_{\nabla_XY}A,
$$
and then the following Ricci commutation formula holds
\begin{equation}\label{eq:C} 
C(X,Y)=C(Y,X)+[R(X,Y),A].
\end{equation}

Now, let $\{e_1,e_2\}$ be an orthonormal basis in
$T_p\Sigma^2$, $p\in\Sigma^2$, and extend the $e_i$'s to vector
fields $X_i$ in a neighborhood of $p$ such that $\{X_1,X_2\}$ is a
geodesic frame field around $p$. Also, denote $X=X_j$. Then we have
$$
(\trace\nabla^2A)X=\sum_{i=1}^2C(X_i,X_i)X.
$$

From equation \eqref{eq:Codazzi}, we get, at $p$,
$$
\begin{array}{ll}
C(X_i,X)X_i&=\nabla_{X_i}((\nabla_{X}A)X_i)\\ \\&=\nabla_{X_i}((\nabla_{X_i}A)X)+2\nabla_{X_i}(\langle \xi,E_3^{\perp}\rangle(\langle X_i,E_3^{\top}\rangle X-\langle X,E_3^{\top}\rangle X_i))
\end{array}
$$
and then, after a straightforward computation, using formulas \eqref{nabla_sol},
\begin{equation}\label{eq:1}
\begin{array}{ll}
C(X_i,X)X_i=&C(X_i,X_i)X\\ \\&+2\{-\langle AX_i,E_3^{\top}\rangle+\langle \xi,E^{\perp}_1\rangle\langle X_i,E^{\top}_1\rangle-\langle\xi,E^{\perp}_2\rangle\langle X_i,E^{\top}_2\rangle\}\\ \\&(\langle X_i,E^{\top}_3\rangle X-\langle X,E^{\top}_3\rangle X_i)\\ \\&+2\langle \xi,E_3^{\perp}\rangle\{(\langle\xi,E_3^{\perp}\rangle\langle AX_i,X_i\rangle+\langle X_i,E_1^{\top}\rangle^2-\langle X_i,E_2^{\top}\rangle^2) X\\ \\&-(\langle\xi,E_3^{\perp}\rangle\langle AX,X_i\rangle+\langle X,E^{\top}_1\rangle\langle X_i,E_1^{\top}\rangle-\langle X,E_2^{\top}\rangle\langle X_i,E_2^{\top}\rangle)X_i\}.
\end{array}
\end{equation}

Denote by $T_iX=C(X_i,X)X_i-C(X_i,X_i)X$, as given by \eqref{eq:1}, and, after another long but quite straightforward computation, using 
$$
|E_k^{\top}|^2=1-|E_k^{\perp}|^2=1-\langle\xi,E_k^{\perp}\rangle^2,\quad\forall k=\overline{1,3},
$$
$$
\langle E_k^{\top},E_l^{\top}\rangle=-\langle E_k^{\perp},E_l^{\perp}\rangle=-\langle\xi,E_k^{\perp}\rangle\langle\xi,E_l^{\perp}\rangle,\quad\forall k,l=\overline{1,3},\quad k\neq l,
$$
and $|\xi|=1$, one obtains
\begin{equation}\label{eq:ti}
\begin{array}{ll}
\sum_{i=1}^{2}T_iX=&\left\{4\langle\xi,E_3^{\perp}\rangle\left(\langle\xi,E_2^{\perp}\rangle^2-\langle\xi,E_1^{\perp}\rangle^2\right)+4f\langle\xi,E_3^{\perp}\rangle^2-2\langle AE_3^{\top},E_3^{\top}\rangle\right\}X\\ \\&+2\langle X,E_3^{\top}\rangle AE_3^{\top}-2\langle\xi,E_3^{\perp}\rangle^2 AX\\ \\&-2(\langle\xi,E_1^{\perp}\rangle\langle X,E_3^{\top}\rangle+\langle\xi,E_3^{\perp}\rangle\langle X,E_1^{\top}\rangle)E_1^{\top}\\ \\&+2(\langle\xi,E_2^{\perp}\rangle\langle X,E_3^{\top}\rangle+\langle\xi,E_3^{\perp}\rangle\langle X,E_2^{\top}\rangle)E_2^{\top}.
\end{array}
\end{equation}

Next, also at $p$, we have $C(X,X_i)X_i=\nabla_{X}((\nabla_{X_i}A)X_i)$
and, from \eqref{eq:C} and \eqref{eq:1}, it follows that
\begin{equation}\label{CXX}
\sum_{i=1}^{2}C(X_i,X_i)X=\sum_{i=1}^{2}(\nabla_{X}((\nabla_{X_i}A)X_i)+[R(X_i,X),A]X_i-T_iX).
\end{equation}

Since $\nabla_{X_i}A$ is symmetric, from \eqref{eq:Codazzi}, one obtains
$$
\begin{array}{lcl}
\langle\sum_{i=1}^2(\nabla_{X_i}A)X_i,Z\rangle&=&\sum_{i=1}^2\langle X_i,(\nabla_{X_i}A)Z\rangle=\sum_{i=1}^2\langle X_i,(\nabla_{Z}A)X_i\rangle\\ \\&&+2\langle\xi,E_3^{\perp}\rangle\sum_{i=1}^2\langle X_i,\langle Z,E_3^{\top}\rangle X_i-\langle X_i,E_3^{\top}\rangle Z\rangle\\ \\&=&\trace(\nabla_ZA)+2\langle\xi,E_3^{\perp}\rangle\langle E_3^{\top},Z\rangle\\ \\&=&Z(\trace A)+2\langle\xi,E_3^{\perp}\rangle\langle E_3^{\top},Z\rangle,
\end{array}
$$
for any vector $Z$ tangent to $\Sigma^2$, and, therefore,
\begin{equation}\label{nablaA}
\sum_{i=1}^2(\nabla_{X_i}A)X_i=2\grad f+2\langle\xi,E_3^{\perp}\rangle E_3^{\top}.
\end{equation}
Next, from equations \eqref{nabla_sol} one obtains $
\bar\nabla_YE_3=\langle Y,E_1^{\top}\rangle E_1-\langle Y,E_2^{\top}\rangle E_2$,
for any tangent vector field $Y$, and, since
$$
\bar\nabla_YE_3=\nabla_YE_3^{\top}+\sigma(Y,E_3^{\top})-\langle\xi,E_3^{\perp}\rangle AY+\nabla^{\perp}_YE_3^{\perp},
$$
we have
\begin{equation}\label{eq:nablaE3}
\nabla_YE_3^{\top}=\langle\xi,E_3^{\perp}\rangle AY+\langle Y,E_1^{\top}\rangle E_1^{\top}-\langle Y,E_2^{\top}\rangle E_2^{\top}.
\end{equation}
Therefore, from \eqref{nablaA}, again using \eqref{nabla_sol}, it follows
\begin{equation}\label{eq:nabla_sum}
\begin{array}{cl}
\sum_{i=1}^{2}\nabla_{X}((\nabla_{X_i}A)X_i)=&2\nabla_X(\grad f)+2\langle\xi,E_3^{\perp}\rangle^2AX\\ \\&+2\langle\xi,E_3^{\perp}\rangle\langle X,E_1^{\top}\rangle E_1^{\top}-2\langle\xi,E_3^{\perp}\rangle\langle X,E_2^{\top}\rangle E_2^{\top}\\ \\&+2(-\langle AX,E_3^{\top}\rangle+\langle\xi,E_1^{\perp}\rangle\langle X,E_1^{\top}\rangle-\langle\xi,E_2^{\perp}\rangle\langle X,E_2^{\top}\rangle)E_3^{\top}.
\end{array}
\end{equation}
Moreover, again using the symmetry of $\nabla_{X_j}A$, we have
\begin{equation}\label{eq:trace_nabla_sum}
\begin{array}{lcl}
2\sum_{j=1}^{2}\langle\nabla_{X_j}(\grad f),AX_j\rangle&=&2\sum_{i=1}^{2}\langle A(\nabla_{X_j}(\grad f)),X_j\rangle\\ \\&=&2\sum_{j=1}^{2}\langle-(\nabla_{X_j}A)(\grad f)+\nabla_{X_j}A(\grad f),X_j\rangle\\ \\&=&2\sum_{j=1}^{2}(\langle-\grad f,(\nabla_{X_j}A)X_j\rangle\\ \\&&+\langle\nabla_{X_j}A(\grad f),X_j\rangle)\\ \\&=&-4|\grad f|^2-4\langle\xi,E_3^{\perp}\rangle\langle\grad f,E_3^{\top}\rangle\\ \\&&+2\Div(A(\grad f)).
\end{array}
\end{equation}

Next, it is easy to verify, by using \eqref{eq:barR} and the Gauss equation of the surface, that the expression $\sum_{i,j=1}^{2}\langle[R(X_i,X_j),A]X_i,AX_j\rangle$ is independent of the choice of the orthonormal frame field on $\Sigma^2$. Thus, we can consider $\{\widetilde X_1,\widetilde X_2\}$ an orthonormal frame field that diagonalizes the shape operator $A$, i.e., $A\widetilde X_i=\lambda_i\widetilde X_i$, and then easily obtain
\begin{equation}\label{eq:R}
\begin{array}{lcl}
\sum_{i,j=1}^{2}\langle[R(X_i,X_j),A]X_i,AX_j\rangle&=&\sum_{i,j=1}^{2}\langle[R(\widetilde X_i,\widetilde X_j),A]\widetilde X_i,A\widetilde X_j\rangle\\ \\&=&-\langle R(\widetilde X_1,\widetilde X_2)\widetilde X_1,\widetilde X_2\rangle(\lambda_1-\lambda_2)^2\\ \\&=&2K(|A|^2-2f^2),
\end{array}
\end{equation}
where $K$ is the Gaussian curvature of the surface.

Finally, from equations \eqref{eq:ti}, \eqref{CXX}, \eqref{eq:nabla_sum}, \eqref{eq:trace_nabla_sum}, and \eqref{eq:R}, we have
$$
\begin{array}{lcl}
\langle\trace\nabla^2A,A\rangle&=&2\Div(A(\grad f))-4|\grad f|^2-4\langle\xi,E_3^{\perp}\rangle\langle\grad f,E_3^{\top}\rangle\\ \\&&+2K(|A|^2-2f^2)+4\langle\xi,E_3^{\perp}\rangle^2(|A|^2-2f^2)\\ \\&&-8f\langle\xi,E_3^{\perp}\rangle(\langle\xi,E_2^{\perp}\rangle^2-\langle\xi,E_1^{\perp}\rangle^2)+4f\langle AE_3^{\top},E_3^{\top}\rangle\\ \\&&+4\langle\xi,E_3^{\perp}\rangle\langle AE_1^{\top},E_1^{\top}\rangle-4\langle\xi,E_3^{\perp}\rangle\langle AE_2^{\top},E_2^{\top}\rangle\\ \\&&+4\langle\xi,E_1^{\perp}\rangle\langle AE_3^{\top},E_1^{\top}\rangle-4\langle\xi,E_2^{\perp}\rangle\langle AE_3^{\top},E_2^{\top}\rangle-4\langle AE_3^{\top},AE_3^{\top}\rangle.
\end{array}
$$

Hence, from \eqref{eq:Laplacian}, we obtain the following result.

\begin{theorem}\label{t:delta} Let $\phi:\Sigma^2\rightarrow\sol$ be a surface of $\sol$. Then we have
\begin{equation}\label{eq:delta2}
\begin{array}{lcl}
\frac{1}{2}\Delta|A|^2&=&|\nabla A|^2+2\Div(A(\grad f))-4|\grad f|^2-4\langle\xi,E_3^{\perp}\rangle\langle\grad f,E_3^{\top}\rangle\\ \\&&+2K(|A|^2-2f^2)+4\langle\xi,E_3^{\perp}\rangle^2(|A|^2-2f^2)\\ \\&&-8f\langle\xi,E_3^{\perp}\rangle(\langle\xi,E_2^{\perp}\rangle^2-\langle\xi,E_1^{\perp}\rangle^2)+4f\langle AE_3^{\top},E_3^{\top}\rangle\\ \\&&+4\langle\xi,E_3^{\perp}\rangle\langle AE_1^{\top},E_1^{\top}\rangle-4\langle\xi,E_3^{\perp}\rangle\langle AE_2^{\top},E_2^{\top}\rangle\\ \\&&+4\langle\xi,E_1^{\perp}\rangle\langle AE_3^{\top},E_1^{\top}\rangle-4\langle\xi,E_2^{\perp}\rangle\langle AE_3^{\top},E_2^{\top}\rangle-4\langle AE_3^{\top},AE_3^{\top}\rangle.
\end{array}
\end{equation}
\end{theorem}

In order to obtain a version of Theorem \ref{t:delta} that will be used in the next section, we need the following two lemmas.

\begin{lemma}\label{lemma1} If $\phi:\Sigma^2\rightarrow\sol$ is a surface as in Theorem \ref{t:delta}, then
\begin{equation}\label{eq:divf}
\begin{array}{lcl}
\Div(f\langle\xi,E_3^{\perp}\rangle E_3^{\top})&=&2f\langle\xi,E_3^{\perp}\rangle(\langle\xi,E_2^{\perp}\rangle^2-\langle\xi,E_1^{\perp}\rangle^2)-f\langle AE_3^{\top},E_3^{\top}\rangle\\ \\&&+2f^2\langle\xi,E_3^{\perp}\rangle^2+\langle\xi,E_3^{\perp}\rangle\langle\grad f,E_3^{\top}\rangle.
\end{array}
\end{equation}
\end{lemma}

\begin{proof} Consider a point $p\in\Sigma^2$ and an orthonormal geodesic frame field $\{X_1,X_2\}$ around $p$ as in the proof of Theorem \ref{t:delta}. 

From equations \eqref{nabla_sol} and \eqref{nablaA}, one obtains
$$
\begin{array}{lcl}
\nabla_{X_i}(\langle\xi,E_3^{\perp}\rangle E_3^{\top})&=&(-\langle AX_i,E_3^{\top}\rangle+\langle\xi,E_1^{\perp}\rangle\langle X_i,E_1^{\top}\rangle-\langle\xi,E_2^{\perp}\rangle\langle X_i,E_2^{\top}\rangle)E_3^{\top}\\ \\&&+\langle\xi,E_3^{\perp}\rangle^2AX_i+\langle\xi,E_3^{\perp}\rangle(\langle X_i,E_1^{\top}\rangle E_1^{\top}-\langle X_i,E_2^{\top}\rangle E_2^{\top}).
\end{array}
$$
Then, since $\{E_k\}_{k=1}^3$ is an orthonormal frame field, $|\xi|=1$, and
$$
\Div(f\langle\xi,E_3^{\perp}\rangle E_3^{\top})=\langle\xi,E_3^{\perp}\rangle\langle\grad f,E_3^{\top}\rangle+f\sum_{i=1}^{2}\langle\nabla_{X_i}(\langle\xi,E_3^{\perp}\rangle E_3^{\top}),X_i\rangle,
$$
we conclude the proof.
\end{proof}

\begin{lemma}\label{lemma2}
If $\phi:\Sigma^2\rightarrow\sol$ is a surface as in Theorem \ref{t:delta}, then
\begin{equation}\label{eq:divA}
\begin{array}{lcl}
\Div(\langle\xi,E_3^{\perp}\rangle AE_3^{\top})&=&-\langle AE_3^{\top},AE_3^{\top}\rangle+2\langle\xi,E_3^{\perp}\rangle\langle\grad f,E_3^{\top}\rangle\\ \\&&+\langle\xi,E_3^{\perp}\rangle^2|A|^2+2\langle\xi,E_3^{\perp}\rangle^2(1-\langle\xi,E_3^{\perp}\rangle^2)\\ \\&&+\langle\xi,E_3^{\perp}\rangle\langle AE_1^{\top},E_1^{\top}\rangle-\langle\xi,E_3^{\perp}\rangle\langle AE_2^{\top},E_2^{\top}\rangle\\ \\&&+\langle\xi,E_1^{\perp}\rangle\langle AE_3^{\top},E_1^{\top}\rangle-\langle\xi,E_2^{\perp}\rangle\langle AE_3^{\top},E_2^{\top}\rangle.
\end{array}
\end{equation}
\end{lemma}

\begin{proof} As before, again consider a point $p\in\Sigma^2$ and an orthonormal geodesic frame field $\{X_1,X_2\}$ around $p$. Since $\nabla_{X_i}A$ are symmetric, we have, also using \eqref{nabla_sol},
$$
\begin{array}{lcl}
\Div(\langle\xi,E_3^{\perp}\rangle AE_3^{\top})&=&\sum_{i=1}^{2}\langle\nabla_{X_i}(\langle\xi,E_3^{\perp}\rangle AE_3^{\top}),X_i\rangle\\ \\&=&\sum_{i=1}^{2}(-\langle AX_i,E_3^{\top}\rangle+\langle\xi,E_1^{\perp}\rangle\langle X_i,E_1^{\top}\rangle\\ \\&&-\langle\xi,E_2^{\perp}\rangle\langle X_i,E_2^{\top}\rangle)\langle AE_3^{\top},X_i\rangle+\langle\xi,E_3^{\perp}\rangle\langle\nabla_{X_i}AE_3^{\top},X_i\rangle)\\ \\&=&-\langle AE_3^{\top},AE_3^{\top}\rangle+\langle\xi,E_1^{\perp}\rangle\langle AE_3^{\top},E_1^{\top}\rangle-\langle\xi,E_2^{\perp}\rangle\langle AE_3^{\top},E_2^{\top}\rangle\\ \\&&+\sum_{i=1}^{2}\langle\xi,E_3^{\perp}\rangle\langle(\nabla_{X_i}A)E_3^{\top}+A(\nabla_{X_i}E_3^{\top}),X_i\rangle\\ \\&=&-\langle AE_3^{\top},AE_3^{\top}\rangle+\langle\xi,E_1^{\perp}\rangle\langle AE_3^{\top},E_1^{\top}\rangle-\langle\xi,E_2^{\perp}\rangle\langle AE_3^{\top},E_2^{\top}\rangle\\ \\&&+\sum_{i=1}^{2}\langle\xi,E_3^{\perp}\rangle(\langle E_3^{\top},(\nabla_{X_i}A)X_i\rangle+\langle A(\nabla_{X_i}E_3^{\top}),X_i\rangle),
\end{array}
$$
and we conclude with equations \eqref{nablaA} and \eqref{eq:nablaE3}.
\end{proof}

Now, from Theorem \ref{t:delta}, together with Lemmas \ref{lemma1} and \ref{lemma2}, we can state the following proposition.

\begin{proposition}\label{p:delta} Let $\phi:\Sigma^2\rightarrow\sol$ be a surface of $\sol$. Then the following equation holds
\begin{equation}\label{eq:delta3}
\begin{array}{lcl}
\frac{1}{2}\Delta|A|^2&=&|\nabla A|^2-4|\grad f|^2-8\langle\xi,E_3^{\perp}\rangle\langle\grad f,E_3^{\top}\rangle\\ \\&&+2\Div(A(\grad f))+4\Div(\langle\xi,E_3^{\perp}\rangle AE_3^{\top})-4\Div(f\langle\xi,E_3^{\perp}\rangle E_3^{\top})\\ \\&&+2K(|A|^2-2f^2)-8\langle\xi,E_3^{\perp}\rangle^2(1-\langle\xi,E_3^{\perp}\rangle^2).
\end{array}
\end{equation}
\end{proposition}
\label{proposition}

When the surface is CMC the above Simons type formula simplifies.

\begin{corollary}
Let $\phi:\Sigma^2\rightarrow\sol$ be a CMC surface of $\sol$. Then we have
\begin{equation}\label{eq:deltaCMC}
\begin{array}{lcl}
\frac{1}{2}\Delta|A|^2&=&|\nabla A|^2+4\Div(\langle\xi,E_3^{\perp}\rangle AE_3^{\top})-4\Div(f\langle\xi,E_3^{\perp}\rangle E_3^{\top})\\ \\&&+2K(|A|^2-2f^2)-8\langle\xi,E_3^{\perp}\rangle^2(1-\langle\xi,E_3^{\perp}\rangle^2).
\end{array}
\end{equation}
\end{corollary}

We end this section with the following two results that provide another Simons type formula.

\begin{proposition}\label{p:delta_angle}
If $\phi:\Sigma^2\rightarrow\sol$ is a surface of $\sol$ as in Theorem \ref{t:delta}, then
\begin{equation}\label{eq:delta_angle}
\begin{array}{lcl}
\frac{1}{2}\Delta\langle\xi,E_3^{\perp}\rangle^2&=&-\Div(f\langle\xi,E_3^{\perp}\rangle E_3^{\top})-\langle\xi,E_3^{\perp}\rangle\langle\grad f,E_3^{\perp}\rangle\\ \\&&+2\langle\xi,E_3^{\perp}\rangle\langle AE_2^{\top},E_2^{\top}\rangle-2\langle\xi,E_3^{\perp}\rangle\langle AE_1^{\top},E_1^{\top}\rangle\\ \\&&+2\langle\xi,E_2^{\perp}\rangle\langle AE_3^{\top},E_2^{\top}\rangle-2\langle\xi,E_1^{\perp}\rangle\langle AE_3^{\top},E_1^{\top}\rangle\\ \\&&+\langle\xi,E_3^{\perp}\rangle^2(2f^2-3-|A|^2)-f\langle AE_3^{\top},E_3^{\top}\rangle+\langle AE_3^{\top},AE_3^{\top}\rangle\\ \\&&+1-(\langle\xi,E_2^{\top}\rangle^2-\langle\xi,E_1^{\top}\rangle^2)^2.
\end{array}
\end{equation}
\end{proposition}

\begin{proof} Let $p\in\Sigma^2$ be a point on the surface and $\{X_1,X_2\}$ a geodesic orthonormal frame field around $p$. Then, we have, from equations \eqref{nabla_sol} and $\nabla^{\perp}\xi=0$,
\begin{equation}\label{eq:grad_angle}
\grad(\langle\xi,E_3^{\perp}\rangle)=-AE_3^{\top}+\langle\xi,E_1^{\perp}\rangle E_1^{\top}-\langle\xi,E_2^{\perp}\rangle E_2^{\top}.
\end{equation}

It follows that
\begin{equation}\label{delta_pre}
\frac{1}{2}\Delta \langle\xi,E_3^{\perp}\rangle^2=\Div(\langle\xi,E_3^{\perp}\rangle(\langle\xi,E_1^{\perp}\rangle E_1^{\top}-\langle\xi,E_2^{\perp}\rangle E_2^{\top}))-\Div(\langle\xi,E_3^{\perp}\rangle AE_3^{\top}).
\end{equation}

In order to compute the first term in the right-hand side part of the above equation, we first note that \eqref{nabla_sol} implies
$$
\bar\nabla_{X_i}E_1=-\langle X_i,E_1^{\top}\rangle E_3,\quad\quad
\bar\nabla_{X_i}E_1=\langle X_i,E_2^{\top}\rangle E_3.
$$
Since the tangent parts of $\bar\nabla_{X_i}E_k$ are, in all four cases, $\nabla_{X_i}E_k^{\top}-\langle X_i,E_k^{\perp}\rangle AX_i$, one can see that
$$
\nabla_{X_i}E_1^{\top}=AE_1^{\top}-\langle X_i,E_1^{\top}\rangle E_3^{\top},\quad\quad
\nabla_{X_i}E_2^{\top}=AE_2^{\top}+\langle X_i,E_2^{\top}\rangle E_3^{\top}.
$$
Then, a straightforward computation, similar to those performed before in this section, together with Lemma \ref{lemma2}, leads to the conclusion.
\end{proof}

From Theorem \ref{t:delta} and Proposition \ref{p:delta_angle} we get the next result.

\begin{proposition}\label{p:combi}
Let $\phi:\Sigma^2\rightarrow\sol$ be a surface of $\sol$. Then the following equation holds
\begin{equation}\label{eq:delta_A_angle}
\begin{array}{lcl}
\frac{1}{2}\Delta(|A|^2+2\langle\xi,E_3^{\perp}\rangle^2)&=&|\nabla A|^2-4|\grad f|^2-2\langle\xi,E_3^{\perp}\rangle\langle\grad f,E_3^{\top}\rangle\\ \\&&+2\Div(A(\grad f))-6\Div(f\langle\xi,E_3^{\perp}\rangle E_3^{\top})\\ \\&&+2K(|A|^2-2f^2)+2\langle\xi,E_3^{\perp}\rangle^2(|A|^2+2f^2-3)\\ \\&&-2\langle AE_3^{\top},AE_3^{\top}\rangle-2f\langle AE_3^{\top},E_3^{\top}\rangle\\ \\&&+2-2(\langle\xi,E_2^{\top}\rangle^2-\langle\xi,E_1^{\top}\rangle^2)^2.
\end{array}
\end{equation}

\end{proposition}

\begin{remark} It can be easily verified, by using an orthonormal frame field which diagonalizes the shape operator $A$, that
$$
2f\langle AE_3^{\top},E_3^{\top}\rangle=\langle AE_3^{\top},AE_3^{\top}\rangle+\frac{1}{2}(1-\langle\xi,E_3^{\perp}\rangle^2)(4f^2-|A|^2),
$$
a formula which can be used to write down some alternative versions of the Laplacians computed in this section.
\end{remark}

\section{Compact CMC surfaces in $\sol$}

We shall apply the formulas developed in the previous section to prove some gap and non-existence results for compact CMC surfaces. As there are no compact minimal surfaces in $\sol$, we shall work only with CMC non-minimal surfaces, i.e., CMC surfaces with $f\neq 0$.

\begin{theorem}\label{thm:compact1} There are no compact CMC surfaces in $\sol$ such that
$$
2f^2+2\leq |A|^2\leq 4f^2-2
$$
throughout the surface.
\end{theorem}

\begin{proof} Let $\phi:\Sigma^2\rightarrow\sol$ be a CMC surface such that 
$$
2f^2+2\leq |A|^2\leq 4f^2-2
$$ 
everywhere on $\Sigma^2$ (which also implies $f^2\geq 2$) and consider the expression 
$$
E=2K(|A|^2-2f^2)-8\langle\xi,E_3^{\perp}\rangle^2(1-\langle\xi,E_3^{\perp}\rangle^2)
$$
that appears in formula \eqref{eq:deltaCMC}. 

Then, from \eqref{eq:Gaussian_curvature}, we have 
$$
\begin{array}{ll}
E&=(4\langle\xi,E_3^{\perp}\rangle^2-2+4f^2-|A|^2)(|A|^2-2f^2)-8\langle\xi,E_3^{\perp}\rangle^2(1-\langle\xi,E_3^{\perp}\rangle^2)\\ \\&=(4f^2-2-|A|^2)(|A|^2-2f^2)+4\langle\xi,E_3^{\perp}\rangle^2(|A|^2-2f^2-2(1-\langle\xi,E_3^{\perp}\rangle^2))\geq 0.
\end{array}
$$

On the other hand, integrating equation \eqref{eq:deltaCMC} over the surface, one sees that 
$$
\int_{\Sigma^2}(|\nabla A|^2+E)d\vol=0,
$$ 
which implies $\nabla A=0$ and $E=0$. It follows that $|A|^2$ is a constant and both terms in the final expression of $E$ vanish. Therefore $|A|^2=4f^2-2$, $\langle\xi,E_3^{\perp}\rangle=0$, and $K=0$.

But $\langle\xi,E_3^{\perp}\rangle=0$ means that $E_3$ is tangent to the surface and, since from \eqref{nabla_sol} we have
$$
\begin{array}{ll}
\bar\nabla_{E_3}E_3&=\nabla_{E_3}E_3+\sigma(E_3,E_3)\\ &=0,
\end{array}
$$
one obtains $\langle AE_3,E_3\rangle=0$. This shows that, if we consider an orthonormal frame field $\{X_1=E_3,X_2\}$ on the surface, we also have $\langle AX_2,X_2\rangle=2f$. Thus, it follows that $|A|^2=4f^2-2\geq 4f^2$ in this case, which is a contradiction.
\end{proof}

Note that in Theorem \ref{thm:compact1} we imposed a condition that allowed only a (quite) rough estimation of the term that we denoted by $E$ in the expression of the Laplacian given by \eqref{eq:deltaCMC}. In order to obtain a sharper result we shall impose a different condition on the surface, combining the Gaussian curvature $K$, the squared norm of $A$ and the mean curvature function $f$. We then have yet another non-existence result.

\begin{theorem}\label{thm:compact2} There are no compact CMC surfaces in $\sol$ such that
\begin{equation}\label{condition1}
K(|A|^2-2f^2)\geq 1
\end{equation}
throughout the surface.
\end{theorem}

\begin{proof} Let us suppose that there exists $\Sigma^2$, a surface satisfying \eqref{condition1} and consider the same term $E$ as in the proof of Theorem \ref{thm:compact1}. Then $E$ satisfies the inequality
$$
E=2K(|A|^2-2f^2)-8\langle\xi,E_3^{\perp}\rangle^2(1-\langle\xi,E_3^{\perp}\rangle^2)\geq 2K(|A|^2-2f^2)-2,
$$
as $\langle\xi,E_3^{\perp}\rangle^2(1-\langle\xi,E_3^{\perp}\rangle^2)\leq 1/4$. Therefore, since our surface satisfies condition \eqref{condition1}, we have $E\geq 0$.

We integrate equation \eqref{eq:deltaCMC} over $\Sigma^2$ and see that $\int_{\Sigma^2}(|\nabla A|^2+E)d\vol=0$, which leads to $\nabla A=0$ and $E=0$. Then, it is easy to see that $K(|A|^2-2f^2)=1$, which, since $|A|^2$ is a constant, implies that $K$ is a constant. Moreover, since all the inequalities we used become equalities, we also have $\langle\xi,E_3^{\perp}\rangle^2=1/2$.

Now, we have $2K=4f^2-|A|^2$ and then $(4f^2-|A|^2)(|A|^2-2f^2)=2$. A simple computation shows that this last algebraic equation has real (and positive) solutions if and only if $f^2\geq\sqrt{2}$ and, in this case, these solutions are
$$
|A|^2=3f^2\pm\sqrt{f^4-2},
$$
with the corresponding values of the Gaussian curvature given by
$$
K=\frac{1}{2}\left(f^2\mp\sqrt{f^4-2}\right).
$$

Next, since $\langle\xi,E_3^{\perp}\rangle^2=1/2$, from equation \eqref{eq:grad_angle} it follows 
$$
AE_3^{\top}=\langle\xi,E_1^{\perp}\rangle E_1^{\top}-\langle\xi,E_2^{\perp}\rangle E_2^{\top}
$$
and, therefore, since $\langle E_k^{\top},E_l^{\top}\rangle=-\langle\xi,E_k^{\perp}\rangle\langle\xi,E_l^{\perp}\rangle$, for $k\neq l$,
$$
\begin{array}{lcl}
\langle AE_3^{\top},AE_3^{\top}\rangle&=&\langle\xi,E_1^{\perp}\rangle^2(1-\langle\xi,E_1^{\perp}\rangle^2)+\langle\xi,E_2^{\perp}\rangle^2(1-\langle\xi,E_2^{\perp}\rangle^2)\\ \\&&+2\langle\xi,E_1^{\perp}\rangle^2\langle\xi,E_2^{\perp}\rangle^2\\ \\&=&\frac{1}{2}-(\langle\xi,E_1^{\perp}\rangle^2-\langle\xi,E_2^{\perp}\rangle^2)^2,
\end{array}
$$
and
$$
\langle AE_3^{\top},E_3^{\top}\rangle=\langle\xi,E_3^{\perp}\rangle(\langle\xi,E_2^{\perp}\rangle^2-\langle\xi,E_1^{\perp}\rangle^2).
$$
Moreover, from Lemma \ref{lemma1}, we obtain, after a straightforward computation,
\begin{equation}\label{eq:divcst}
\Div(f\langle\xi, E_3^{\perp}\rangle E_3^{\top})=f^2+f\langle\xi,E_3^{\perp}\rangle(\langle\xi,E_1^{\perp}\rangle^2-\langle\xi,E_2^{\perp}\rangle^2).
\end{equation}

Replacing into equation \eqref{eq:delta_A_angle} and again taking into account that $|A|^2$ and $\langle\xi,E_3^{\perp}\rangle$ are constants, it follows that
\begin{equation}\label{eq:suppl}
|A|^2-4f^2-8f\langle\xi,E_3^{\perp}\rangle(\langle\xi,E_1^{\perp}\rangle^2-\langle\xi,E_2^{\perp}\rangle^2)=0.
\end{equation}
Integrating \eqref{eq:divcst} and \eqref{eq:suppl} over $\Sigma^2$, one obtains $|A|^2=f^2=0$, which is a contradiction.
\end{proof}

Finally, using both Laplacians given by \eqref{eq:deltaCMC} and \eqref{eq:delta_A_angle} and imposing the most general (in this situation) condition on $E$, we get a general non-existence result.

\begin{theorem}\label{thm:compact3} There are no compact CMC surfaces in $\sol$ with
\begin{equation}\label{eq:cond_gen}
4K^2+8K|A|^2+|A|^4-24f^2K-8f^2|A|^2+16f^4-4\geq 0
\end{equation}
throughout the surface.
\end{theorem}

\begin{proof} A simple computation, using $2\langle\xi,E_3^{\perp}\rangle=K+1-2f^2+(1/2)|A|^2$, gives
$$
\begin{array}{lcl}
E&=&2K(|A|^2-2f^2)-8\langle\xi,E_3^{\perp}\rangle^2(1-\langle\xi,E_3^{\perp}\rangle^2)\\ \\&=&8(4K^2+8K|A|^2+|A|^4-24f^2K-8f^2|A|^2+16f^4-4)\geq 0,
\end{array}
$$
and, as in the proofs of the last two results, it follows that $\nabla A=0$, which implies that $|A|^2$ is a constant, and $E=0$, which then shows that also $K$ and therefore $\langle\xi,E_3^{\perp}\rangle$ are constants.

Now, since $\grad(\langle\xi,E_3^{\perp}\rangle)=0$, one obtains
$$
AE_3^{\top}=\langle\xi,E_1^{\perp}\rangle E_1^{\top}-\langle\xi,E_2^{\perp}\rangle E_2^{\top},
$$
from where we get
$$
\langle AE_3^{\top},AE_3^{\top}\rangle=1-\langle\xi,E_3^{\perp}\rangle-(\langle\xi,E_1^{\perp}\rangle^2-\langle\xi,E_2^{\perp}\rangle^2)^2,
$$
$$
\langle AE_3^{\top},E_3^{\top}\rangle=\langle\xi,E_3^{\perp}\rangle(\langle\xi,E_2^{\perp}\rangle^2-\langle\xi,E_1^{\perp}\rangle^2),
$$
and
$$
\Div(f\langle\xi, E_3^{\perp}\rangle E_3^{\top})=f^2+f\langle\xi,E_3^{\perp}\rangle(\langle\xi,E_1^{\perp}\rangle^2-\langle\xi,E_2^{\perp}\rangle^2).
$$

Replacing into equation \eqref{eq:delta_A_angle} and, since $|A|^2$ and $\langle\xi,E_3^{\perp}\rangle$ are constants, it follows that
$$
\begin{array}{rcr}
2K(|A|^2-2f^2)-8f^2\langle\xi,E_3^{\perp}\rangle^2+2|A|^2\langle\xi,E_3^{\perp}\rangle^2-4\langle\xi,E_3^{\perp}\rangle^2&&\\ \\-8f\langle\xi,E_3^{\perp}\rangle(\langle\xi,E_1^{\perp}\rangle^2-\langle\xi,E_2^{\perp}\rangle^2)&=&0.
\end{array}
$$

From here, we either have that $\langle\xi,E_3^{\perp}\rangle=0$ and $K(|A|^2-2f^2)=0$, or that $\langle\xi,E_1^{\perp}\rangle^2-\langle\xi,E_2^{\perp}\rangle^2$ is a constant. In the later case, all $\langle\xi,E_i^{\perp}\rangle$ are constants, which is not possible for compact CMC surfaces, as shown by \cite[Theorem~4.2]{LM}. In the first case $E_3$ is tangent to the surface and, as we have seen in the proof of Theorem~\ref{thm:compact1}, this implies $|A|^2\geq 4f^2$. Therefore, we must have $K=0$, but, from equation \eqref{eq:Gaussian_curvature}, one obtains $|A|^2=4f^2-2< 4f^2$, which is a contradiction.
\end{proof}

We end this section with a result that can be proved exactly like Theorem \ref{thm:compact3}.

\begin{theorem} There are no compact CMC surfaces in $\sol$ with $|A|^2$ a constant and satisfying
$$
4K^2+8K|A|^2+|A|^4-24f^2K-8f^2|A|^2+16f^4-4\leq 0
$$
throughout the surface.
\end{theorem}

\end{document}